%% file: hedgehog_main.tex
\documentclass[12pt]{article}

\usepackage{mathtools,graphicx, mathtools, amsthm, tikz, fancyhdr,hyperref}
\usepackage[toc,page]{appendix}

\mathtoolsset{showonlyrefs}

\hypersetup{
    colorlinks=true, 
    linkcolor=blue, 
    urlcolor=red, 
    linktoc=all 
}

\theoremstyle{plain}
\newtheorem{theorem}{Theorem}[section]
\newtheorem{lemma}[theorem]{Lemma}

\theoremstyle{definition}
\newtheorem{definition}[theorem]{Definition}

\newcommand\Erdos{Erd\H{o}s }
\newcommand\Rodl{R\"odl }

\providecommand{\abs}[1]{\left\vert#1\right\vert}
\newcommand{\defeq}{\vcentcolon=}
\newcommand{\ind}[1]{^{(#1)}}
\renewcommand{\Pr}{\mathop{\bf Pr\/}}
\newcommand{\E}{\mathop{\bf E\/}}
\newcommand{\Var}{\mathop{\bf Var\/}}

\begin{document}

\title{On Ramsey numbers of hedgehogs}
\author{Jacob Fox\thanks{Department of Mathematics, Stanford University, Stanford, CA 94305. Email: {\tt jacobfox@stanford.edu}. Research supported by a Packard Fellowship, and by NSF Career Award DMS-1352121.}\,\, and Ray Li\thanks{Department of Computer Science, Stanford University, Stanford, CA 94305. Email: {\tt rayyli@cs.stanford.edu}.  Research supported by the National Science Foundation Graduate Research Fellowship Program under Grant No. DGE-1656518.}}
\date{\today}
\maketitle

\begin{abstract}
  The hedgehog $H_t$ is a 3-uniform hypergraph on vertices $1,\dots,t+\binom{t}{2}$ such that, for any pair $(i,j)$ with $1\le i<j\le t$, there exists a unique vertex $k>t$ such that $\{i,j,k\}$ is an edge.
  Conlon, Fox, and \Rodl proved that the two-color Ramsey number of the hedgehog grows polynomially in the number of its vertices, while the four-color Ramsey number grows exponentially in the number of its vertices. 
  They asked whether the two-color Ramsey number of the hedgehog $H_t$ is  nearly linear in the number of its vertices.
  We answer this question affirmatively, proving that $r(H_t) = O(t^2\ln t)$.
\end{abstract}


\input{introduction}

\input{hedgehog_proof}

\end{document}

%% file: introduction.tex
\section{Introduction}

For a $k$-uniform hypergraph $H$, the Ramsey number $r(H)$ is the smallest $n$ such that any 2-coloring of $K_n\ind{k}$, the complete $k$-uniform hypergraph on $n$ vertices, contains a monochromatic copy of $H$.
Let $r(H; q)$ denote the analogous Ramsey number for $q$-colorings, so that $r(H) = r(H; 2)$.

It is a major open problem to determine the growth of $r(K_t\ind{3})$, the Ramsey number of the complete 3-uniform hypergraph on $t$ vertices. It is known \cite{EHR65, ER52} that there are constants $c,c'>0$ such that
\begin{align}
  2^{ct^2} \le r(K_t\ind{3}) \le 2^{2^{c't}}.
\label{}
\end{align}
\Erdos conjectured that $r(K_t\ind{3})=2^{2^{\Theta(t)}}$, i.e. the upper bound is closer to the truth. 
\Erdos and Hajnal gave some evidence that this conjecture is true by showing that $ r_3(K_t\ind{3}; 4) \ge 2^{2^{ct}}$, i.e. the four color Ramsey number of $K_t\ind{3}$ is double-exponential in $t$ (see, for example \cite{GRS90}).

\begin{definition}
  The \emph{hedgehog} $H_t$ is a 3-uniform hypergraph on $t+\binom{t}{2}$ vertices $1,\dots,t+\binom{t}{2}$ such that, for each $1\le i<j\le t$, there exists a unique vertex $k>t$ such that $\{i,j,k\}$ is an edge, and there are no additional edges.
\end{definition}
We sometimes refer to the first $t$ vertices as the \emph{body} of the hedgehog.
For any $k\ge 4$, one can also define a $k$-uniform hedgehog $H_t\ind{k}$ on $t+\binom{t}{k-1}$, with a body of size $t$ and a unique hyperedge for every $k-1$-sized subset of the body.
In this notation, we have $H_t=H_t\ind{3}$.

Hedgehogs are interesting because their 2-color Ramsey number $r(H_t; 2)$ is polynomial in $t$, while their 4-color Ramsey number $r(H_t; 4)$ is exponentially large in $t$ \cite{KR06, CFR15}.
This suggests that the bound $r(K_t\ind{3}; 4)\ge 2^{2^{ct}}$ by \Erdos and Hajnal may not be such strong evidence that $r(K_t\ind{3})=2^{2^{\Theta(t)}}$.

Hedgehogs are also interesting because they are a natural family of hypergraphs with \emph{degeneracy} 1. 
Degeneracy is a notion of sparseness for graphs and hypergraphs. 
For graphs, the degeneracy is defined as the minimum $d$ such that every subgraph induced by a set of vertices has a vertex of degree at most $d$. 
The Burr-\Erdos conjecture \cite{BE75} states that there exists a constant $c(d)$ depending only on $d$ such that the Ramsey number of any $d$-degenerate graph $G$ on $n$ vertices satisfies $r(G) \le c(d)\cdot n$.
Building on the work of Kostochka and Sudakov \cite{KS03} and Fox and Sudakov \cite{FS09}, Lee \cite{L15} recently proved this conjecture.
We can similarly define the degeneracy of a hypergraph as the minimum $d$ such that every subhypergraph induced by a subset of vertices has a vertex of degree at most $d$.
Under this definition, Conlon, Fox, and \Rodl \cite{CFR15} observe that the 4-uniform analogue of the Burr-\Erdos conjecture is false: the 4-uniform hedgehog $H\ind{4}_t$, which is 1-degenerate, satisfies $r(H\ind{4}_t)\ge 2^{ct}$.
They also observe that the 3-uniform analogue of the Burr-\Erdos conjecture is false for 3 or more colors: the 3-uniform hedgehog, which is 1-degenerate, satisfies $r(H_t; 3) \ge \Omega(t^3/\log^6t)$.

However, the analogue of the Burr-\Erdos conjecture for 3-uniform hypergraphs and 2 colors remains open.
In particular, it was not known whether the Ramsey number of the hedgehog $H_t$ is linear, or even near-linear, in the number of vertices, $t+\binom{t}{2}$.
Conlon, Fox and Rodl \cite{CFR15} show $r(H_t; 2)\le 4t^3$, and, with the above in mind, ask if $r(H_t; 2) = t^{2+o(1)}$.
We answer this question affirmatively.
\begin{theorem}
\label{thm:main}
If $t\ge 10$ and  $n\ge 200t^2\ln t + 400t^2$, then every two-coloring of the complete $3$-uniform hypergraph on vertices contains a monochromatic copy of the hedgehog $H_t$. That is,
\begin{align*}
  r(H_t) < 200t^2\ln t + 400t^2+1.
\end{align*}
\end{theorem}
We make no attempt to optimize the absolute constants here.

%% file: hedgehog_proof.tex
\usetikzlibrary{positioning}

\section{Ramsey number of hedgehogs}

Throughout this section, we assume $t\ge 10$, and that we have a fixed two-coloring of the edges of a complete 3-uniform hypergraph $\mathcal{H}$ on vertex set $V$ with $n\ge 200t^2\ln t + 400t^2$ vertices.
Let
\begin{align}
  m_{max}\defeq 2t + \binom{t}{2}.
\end{align}
Let $\binom{S}{2}$ denote the set of pairs of elements of $S$.
For integer $a$, let $[a]=\{1,2,\dots,a\}$.
For vertices $u$ and $v$ of $\mathcal{H}$, we write $uv$ as an abbreviation for the unordered pair $\{u,v\}$.

For $u,v\in V$, let
\begin{align}
  d\ind{r}_{uv} \ &\defeq \   \abs{\{w:\{u,v,w\}\text{ red}\}} \nonumber\\
  d\ind{b}_{uv} \ &\defeq \   \abs{\{w:\{u,v,w\}\text{ blue}\}}.
\end{align}
For a set of pairs $F\subset \binom{V}{2}$, let
\begin{align}
  N\ind{b}(F) \ &\defeq \  \left\{ w:\exists uv\in F \text{ s.t. } \{u,v,w\} \text{ blue} \right\} \nonumber\\
  N\ind{r}(F) \ &\defeq \  \left\{ w:\exists uv\in F \text{ s.t. } \{u,v,w\} \text{ red} \right\}.
\label{}
\end{align}
Here, and throughout, we use $b$ and $r$ to refer to the colors blue and red, respectively.
For a vertex $v$ and set $X$, let
\begin{align}
  U\ind{b}_{\le m}(v, X) \ &= \   \left\{ u \in X: d\ind{r}_{uv} \le m \right\} \nonumber\\
  U\ind{r}_{\le m}(v, X) \ &= \   \left\{ u \in X: d\ind{b}_{uv} \le m \right\}.
\end{align}
If $X$ is omitted, take $X=V$.
We define $U\ind{b}_{\le m}(v,X)$ to be sets of $u$ such that $d\ind{r}_{uv}$ is small, rather than those such that $d\ind{b}_{uv}$ is small, because we wish to think of $U\ind{b}$'s as sets helpful for finding a blue hedgehog.
Similarly, we think of $U\ind{r}$'s as sets helpful for finding a red hedgehog.

\begin{lemma}
  \label{lem:alg-0}
  For any $0\le m< \frac{|V|}{2}-1$, and $v\in V$,
  \begin{align}
    \min\left( |U\ind{b}_{\le m}(v)|, |U\ind{r}_{\le m}(v)| \right) \le 2m.
  \label{}
  \end{align}
\end{lemma}
\begin{proof}
  Fix $m$ and $v$.
  For convenience, let $A=U\ind{b}_{\le m}(v)$ and $B=U\ind{r}_{\le m}(v)$.
  Assume for contradiction that $|A|, |B| \ge 2m + 1$.
  For every $u$, we have $d\ind{r}_{uv}+d\ind{b}_{uv} = |V|-2 > 2m$, so $A$ and $B$ are disjoint.
  Consider the set $E'$ of edges of $\mathcal{H}$ containing $v$, one element of $A$, and one element of $B$.
  On one hand, $|E'|=|A|\cdot |B|$.
  On the other hand, for every $u\in A$, the pair $uv$ is in at most $m$ such red triples, so the number of red triples of $E'$ is at most $|A|\cdot m$.
  Additionally, for every $u\in B$, the pair $uv$ is in at most $m$ such blue triples, so the number of blue triples of $E'$ is at most $|B|\cdot m$.
  Hence, $(|A|+|B|)\cdot m \le |E'| = |A|\cdot |B|$, a contradiction of $|A|,|B|\ge 2m+1$.
\end{proof}

The following ``matching condition'' for hedgehogs is useful.
\begin{lemma}
  \label{lem:hall}
  Let $S\subset V$ be a set of $t$ vertices.
  If, for all nonempty sets $F\subset \binom{S}{2}$, we have $|N\ind{b}(F)|\ge |F| + t$, then there exists a blue hedgehog with body $S$.
  Similarly, if, for all nonempty sets $F\subset\binom{S}{2}$, we have $|N\ind{r}(F)|\ge |F| + t$, then there exists a red hedgehog with body $S$.
\end{lemma}
\begin{proof}
  By symmetry, it suffices to prove the first part.
  Consider the bipartite graph $G$ between pairs in $\binom{S}{2}$ and vertices of $V\setminus S$, where $uv\in \binom{S}{2}$ is connected with $w\in V\setminus S$ if and only if triple $\{u,v,w\}$ is blue.
  If, for all nonempty $F \subset\binom{S}{2}$, we have $|N\ind{b}(F)|\ge |F|+t$, then any such $F$ has at least $|F|+t - |S| = |F|$ neighbors in $G$.
  By Hall's marriage lemma on $G$, there exists a matching in $G$ using every element of $\binom{S}{2}$.
  Taking triples $\{u,v,w\}$ where $uv\in\binom{S}{2}$ and $w\in V\setminus S$ is the vertex matched with pair $uv$ gives a blue hedgehog with body $S$.
\end{proof}

\subsection{Special Cases}

We start by finding monochromatic hedgehogs in two specific classes of colorings on $\mathcal{H}$.
We base our proof of Theorem~\ref{thm:main} on the argument for the first class of colorings, which we call \emph{simple colorings}.
We use the result for the second class of colorings, which we call \emph{balanced colorings}, as a specific case in the general argument.

\subsubsection{Simple colorings}
Consider hypergraphs that are colored the following way:
\begin{enumerate}
\item Start with a graph $G$ on $[n]$.
\item Color a complete hypergraph $\mathcal{H}$ on $[n]$ by coloring the triple $\{u,v,w\}$ blue if at least one of $uv,uw,vw$ is in $G$, and red otherwise.
\end{enumerate}
\begin{lemma}
  If $n\ge t^2 + t$, any hypergraph colored as above has a monochromatic $H_t$.
\label{lem:motiv-1}
\end{lemma}
\begin{proof}
  Set $X=V(G)$.
  For $i=t-1,t-2,\dots,0$, pick a vertex $v_i\in X$ whose degree in $G$ is at least $i$ and let $\hat U(v_i) \subset X$ be an arbitrary set of $i$ neighbors of $v_i$.
  Remove $v_i\cup \hat U(v_i)$ from $X$.
  We call this the \emph{peeling step} of $v_i$.
  Figure~\ref{fig:1} shows the first three peeling steps of this process for $t=5$.
  If this process succeeds, we have found a set $S=\{v_{t-1},\dots,v_0\}$ of $t$ vertices and disjoint sets of vertices $\hat U(v_0),\dots, \hat U(v_{t-1})$ also disjoint from $S$, from which we can greedily embed a blue-hedgehog in $\mathcal{H}$ with body $\{v_0,\dots,v_{t-1}\}$: for each $v_iv_j$ with $i < j$, pick an arbitrary unused element of $\hat U(v_j)$ for the third vertex of the hedgehog's edge containing $v_iv_j$.

  Now suppose this process finds vertices $v_{t-1},v_{t-2},\dots,v_{i+1}$ but fails to find $v_i$ for some $i\le t-1$.
  After picking $v_j$, we remove $v_j$ and $j$ of it's neighbors from $X$, for a total of $j+1$ vertices.
  Then we have removed exactly $t+(t-1)+\cdots+(i+2) = \binom{t+1}{2} - \binom{i+2}{2}$ vertices from $X$.
  Hence, $|X|\ge (t^2+t) - \binom{t+1}{2} + \binom{i+2}{2} = \binom{t+1}{2} + \binom{i+2}{2} > \frac{t^2+i^2}{2} \ge ti$, and every vertex has degree at most $i-1$ in the subgraph of $G$ induced $X$.
  Thus, there exists an independent set $S\subset X$ in $G$ of size at least $|X|/i\ge t$.
  Furthermore, any vertex has at most $i-1$ neighbors in $X$, so any two vertices $u,v\in S$ share at least $|X|-2i\ge t+\binom{t}{2} + \binom{i+2}{2}-2i > t + \binom{t}{2}$ red triples in the subhypergraph of $\mathcal{H}$ induced by $X$, so we can greedily find a red hedgehog with body $S$.
\end{proof}
  \begin{figure}
    \label{fig:1}
    \centerline{
      \includegraphics[height=150px]{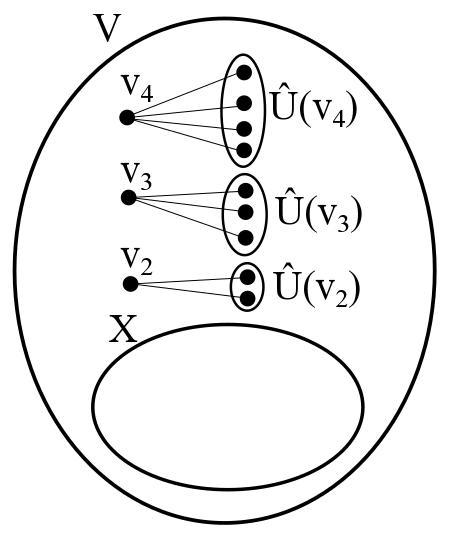}
    }
    \caption{Peeling $v_4, v_3, v_2$ in Lemma~\ref{lem:motiv-1}}
  \end{figure}

\subsubsection{Balanced colorings}

In this section, we consider the case where our coloring is ``balanced''. 
Lemma~\ref{lem:alg-0} tells us that, for every vertex $v$ and every nonnegative integer $m$ less than $\frac{|V|}{2}-1$, one of $|U_{\le m}\ind{b}(v)| = \#\{u:d\ind{r}_{uv}\le m\}$ and $|U_{\le m}\ind{r}(v)|=\#\{u:d\ind{b}_{uv}\le m\}$ is at most $2m$.
In ``balanced'' colorings, we assume, for all $v\in V$ and all $2t\le m\le m_{max}\defeq 2t+\binom{t}{2}$, \emph{both} of $|U_{\le m}\ind{b}(v)|$ and $|U_{\le m}\ind{r}(v)|$ are $O(m)$.
We show, in this case, there is a monochromatic hedgehog.
The proof is by choosing a random subset of approximately $4t$ vertices, and showing that, with positive probability, we can remove vertices so that the remaining set of $t$ vertices is the body of some red hedgehog.
\begin{lemma}
  Let $c\ge 1$.
  Consider a two-colored hypergraph $\mathcal{H}=(V,E)$ on $n\ge 40ct^2$ vertices.
  Suppose that for all $2t\le m\le m_{max}$ and all $v\in V$, we have
  \begin{align}
    \abs{U\ind{b}_{\le m}(v)}\le cm.
  \label{eq:bal-0}
  \end{align}
  Then $\mathcal{H}$ has a red hedgehog $H_t$.
\label{lem:bal}
\end{lemma}
\begin{proof}
  It suffices to prove for $n=40ct^2$, so assume without loss of generality that $n=40ct^2$.
  Pick a random set $S$ by including each vertex of $V$ in $S$ independently with probability $4t/n$.
  By the Chernoff bound, $\Pr[|S|\le 3t] \le e^{-t/8}$. 

  Fix $m$ such that $2t\le m\le m_{max}$ and $m$ is a multiple of $t$.
  Let $e_1,\dots,e_p$ be the pairs such that $d\ind{b}_{e_\ell}\le m$ for all $\ell\in[p]$, and let $X_1,\dots,X_p$ the indicator random variables for these pairs being in $\binom{S}{2}$.
  Let $X=X_1+\cdots+X_p$.
  By \eqref{eq:bal-0}, we have $p\le cmn/2$.
  Each $X_\ell$ for $\ell\in[p]$ is a Bernoulli$(16t^2/n^2)$ random variable.
  Consider a graph on $[p]$ where $\ell$ and $\ell'$ are adjacent (written $\ell\sim \ell'$) if $e_\ell$ and $e_{\ell'}$ share a vertex.
  This is a valid dependency graph for $\{X_{\ell}\}$ as $X_{\ell}$ is independent of all $X_{\ell'}$ such that $e_{\ell'}$ is vertex disjoint from $e_{\ell}$.
  Furthermore, by the condition \eqref{eq:bal-0}, each endpoint of any pair $e_{\ell}$ is in at most $cm$ pairs, so each $\ell\in[p]$ has degree at most $2c m$ in the dependency graph, and the total number of pairs $(\ell,\ell')$ such that $\ell\sim \ell'$ is at most $2cmp$.
  We have
  \begin{align}
    \label{eq:bal-1-1}
    \E[X] \ &= \ \frac{16 t^2 p}{n^2} \ = \ \frac{2p}{5cn}  \ \le \ \frac{m}{5} \ < \ \frac{3m}{4} - t, \\
    \Var[X] 
    \ &= \ \sum_{\ell,\ell'\in[p]}^{} \E[X_{\ell}X_{\ell'}]-\E[X_{\ell}]\E[X_{\ell'}] \nonumber\\
    \ &= \  \sum_{\ell\sim \ell'}^{} \E[X_{\ell}X_{\ell'}] - \E[X_{\ell}]\E[X_{\ell'}] \nonumber\\
    \ &\le \ 2cmp\cdot \left( \left( \frac{4t}{n} \right)^3 - \left( \frac{4t}{n} \right)^4 \right) \nonumber\\ 
    \ &< \ \frac{128t^3c mp}{n^3}
    \ \le \ \frac{64t^3c^2m^2}{n^2} 
    \ = \ \frac{m^2}{25t}. 
    \label{eq:bal-1-2}
  \end{align}
  Hence,
  \begin{align}
    \Pr\left[\#\left\{uv\in \tbinom{S}{2}:d\ind{r}_{uv}\le m\right\}> m-t\right]
    \ &= \  \Pr[X > m-t] \nonumber\\
    \ &= \  \Pr[X-\E[X]\ge m-t-\E[X]] \nonumber\\
    \ &\le \  \Pr[X-\E[X]\ge m/4] \nonumber\\
    \ &\le \  \frac{\Var[X]}{(m/4)^2}
    \ < \ \frac{16}{25t}. 
  \label{}
  \end{align}
  The first inequality is by \eqref{eq:bal-1-1} and the second is by Chebyshev's inequality.
  By the union bound over the multiples of $t$ in $[2t,m_{max}]$, of which there are less than $t$, the probability there exists some $m\in[2t,m_{max}]$ a multiple of $t$ with
  \begin{align}
    \#\left\{uv\in \tbinom{S}{2}:d\ind{r}_{uv}\le m\right\}\le m-t
  \label{eq:bal-2.5}
  \end{align}
  is less than $t\cdot \frac{16}{25t} = \frac{16}{25}$.
  Again by the union bound, with probability more than $1-(\frac{16}{25}+e^{-t/8}) > 0$ over the randomness of $S$, we have (i) $|S|\ge 3t$, and (ii) for all $m$ a multiple of $t$ in $[2t,m_{max}]$, \eqref{eq:bal-2.5} holds.
  Hence, there exists an $S$ such that (i) and (ii) hold, so consider such an $S$.
  Remove $|S|-t \ge 2t$ vertices from $S$, at least one from each of the $2t$ pairs with smallest $d\ind{r}_{uv}$, to obtain a set of $t$ vertices $T$ such that, for all $m$ a multiple of $t$ in $[2t,m_{max}]$, we have
  \begin{align}
    \#\left\{uv\in \tbinom{T}{2}:d\ind{r}_{uv}\le m\right\}
    \le \max\left(0, \#\left\{uv\in \tbinom{S}{2}:d\ind{r}_{uv}\le m\right\} - 2t\right)
    \le \max(0, m - 3t).
  \label{}
  \end{align}
  Then, for all $m$ with $2t\le m\le m_{max}-t$, set $m'$ to be the smallest multiple of $t$ larger than $m$, so that
  \begin{align}
    \#\left\{uv\in \tbinom{T}{2}:d\ind{r}_{uv}\le m\right\}
    \le \#\left\{uv\in \tbinom{T}{2}:d\ind{r}_{uv}\le m'\right\}
    \le \max(0, m'-3t)
    \le m - 2t.
  \label{eq:bal-3}
  \end{align}
  Now, we show our matching condition holds.
  Setting $m=2t$ in \eqref{eq:bal-3}, we have $x_{uv} > 2t$ for all $uv\in\binom{T}{2}$.
  Hence, for any nonempty subset $F\subset\binom{T}{2}$ of size at most $t$, any $uv\in F$ satisfies $x_{uv}>t+|F|$.
  If $F\subset\binom{T}{2}$ has size greater than $t$, then, by setting $m=t+|F|$ in \eqref{eq:bal-3}, we know that there are at most $m-2t=|F|-t$ pairs $uv\in F$ such that $d\ind{r}_{uv}\le t+|F|$, so again there exists $uv\in F$ such that $d\ind{r}_{uv} > t+|F|$.
  We conclude that, for all nonempty subsets of pairs $F\subset\binom{T}{2}$, there exists $uv\in F$ such that $|N\ind{r}(F)|\ge d\ind{r}_{uv} \ge t+|F|$.
  By Lemma~\ref{lem:hall}, there exists a red hedgehog with body $T$.
\end{proof}

\subsection{Proof of Theorem~\ref{thm:main}}
\label{ssec:alg}

\subsubsection{Proof outline}

To prove Theorem~\ref{thm:main}, we follow the proof of Lemma~\ref{lem:motiv-1}.
First, ``peel off'' vertices $v$ into a set $S$ to try to find a blue or red hedgehog.\footnote{
For technical reasons, we peel vertices to find both blue and red hedgehogs, as opposed to Lemma~\ref{lem:motiv-1} where we only peeled vertices to find a blue hedgehog.}
If we succeed, we are done.
If we fail, we end up with an induced two-colored hypergraph that is ``balanced'' in the sense of Lemma~\ref{lem:bal}.
In this case, we simply apply Lemma~\ref{lem:bal}.

In the proof of Lemma~\ref{lem:motiv-1}, we started with $X=V$ and iteratively removed from $X$ a vertex $v$ and a set $\hat U(v)$ of size $t$ such that, for all $u\in \hat U(v)$, vertices $u$ and $v$ share many blue triples.
This deletes $O(t)$ vertices per round, which is small enough for the argument to succeed.
For general hypergraphs, we peel off vertices $v$ with many ``blue-heavy neighbors'', meaning there exists some $m$ such that $|U\ind{b}_{\le m}(v, X)|\ge 10m$.\footnote{
For technical reasons, we peel vertices $v$ in increasing order of the corresponding $m$.}
However, $m$ can be $\Theta(t^2)$, so if we simply deleted $v$ along with $10m$ of its blue-heavy neighbors $\hat U\ind{b}(v)\subset U\ind{b}_{\le m}(v, X)$, we could delete $\Theta(t^2)$ vertices for every $v$, which is too many.
Instead, when we peel off $v$, we delete $v$ from $X$, add a \emph{penalty} of $t/m$ to each $u\in \hat U\ind{b}(v)$, accumulated as $\alpha\ind{b}(u)$, and delete from $X$ every vertex $u$ with $\alpha\ind{b}(u)\ge 1/2$. 
With these penalties, we guarantee that, on average, we delete $O(t)$ vertices from $X$ per peeled vertex $v$.

However, we need more care.
In Lemma~\ref{lem:motiv-1}, we can find a hedgehog with body $S$ because, for any peeled vertices $v,v'\in S$, the edges $\{u,v,v'\}$ are blue for \emph{every} $u\in\hat U(v)$. 
However, in our procedure, for a $v$ chosen with corresponding $\hat U\ind{b}(v)$ of size $10m$, there are some vertices $w$ such that $\{u,v,w\}$ is blue for few (at most $4m$) vertices $u\in\hat U\ind{b}(v)$.
We denote this set of ``bad'' vertices by $B\ind{b}(v)$.
As much as possible, we wish to avoid choosing both $v$ and, at some later step, $w\in B\ind{b}(v)$ for the body $S\ind{b}$ of our blue hedgehog.
Ideally, we simply delete all vertices $u\in B(v)$ in the step we peel off $v$.
However, $B\ind{b}(v)$ can have $\Omega(m)$ vertices, which again could be too many if $m=\Theta(t^2)$.
Instead, for each $w\in B\ind{b}(v)$ we add a \emph{penalty} of $t/d\ind{b}_{wv}$, accumulated as $\beta\ind{b}(w)$, and delete from $X$ every vertex $w$ with $\beta\ind{b}(w)\ge 1/4$.
We guarantee that, on average, we delete $O(t\ln t)$ vertices from $X$ per peeled vertex $v$ (Lemma~\ref{lem:alg-3}).

\begin{figure}
    \label{fig:2}
    \centerline{
      \includegraphics[height=150px]{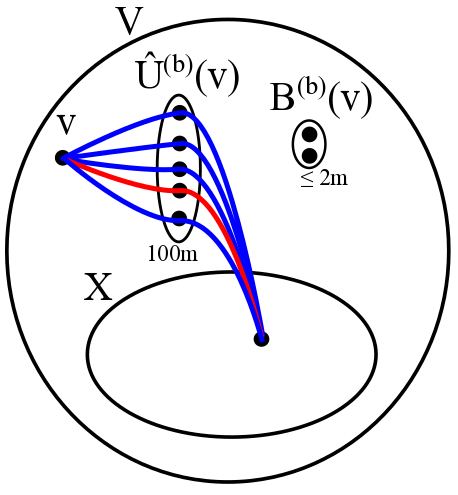}
      \qquad
      \qquad
      \includegraphics[height=150px]{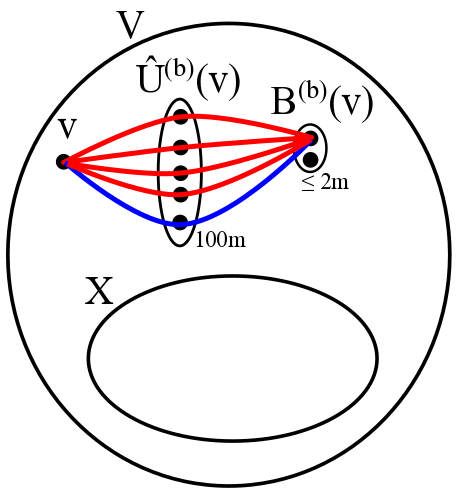}
    }
    \caption{Peeling $v$ with many blue-heavy neighbors. For every $w\in X$, edge $\{u,v,w\}$ is blue for many $u\in\hat U\ind{b}(v)$. Vertices $w\in B\ind{b}(v)$ are the exception. Ideally we simply delete vertex $v$, set $\hat U\ind{b}(v)$, and set $B\ind{b}(v)$ from $X$ (depicted), but instead we maintain fractional penalties $\alpha\ind{\chi}(\cdot)$ and $\beta\ind{\chi}(\cdot)$. We have $|\hat U\ind{b}(v)|=10m$ by definition, and $|B\ind{b}(v)|\le 2m$ by Lemma~\ref{lem:alg-2}.}
  \end{figure}

To finish the proof, we show, if our peeling produces a set $S\ind{b}=\{v_1,\dots,v_t\}$ (where $v_i$ is chosen before $v_{i+1}$), then, because we track the penalties $\alpha\ind{b}(u)$ and $\beta\ind{b}(w)$ carefully, the matching condition of Lemma~\ref{lem:hall} holds.
On the other hand, if the peeling procedure fails, the subhypergraph induced by $X$ is large and balanced, in which case we apply Lemma~\ref{lem:bal}.

\subsubsection{The peeling procedure}

We now describe the procedure formally.
Start with $S\ind{b}=S\ind{r}=\emptyset$, and $X=V$.
For all $u\in V$, initialize $\alpha\ind{r}(u) = \alpha\ind{b}(u) = \beta\ind{r}(u) = \beta\ind{b}(u) = 0$.
If, at any point, $S\ind{b}$ or $S\ind{r}$ has $t$ vertices, stop.

\newcommand\Stage{\textnormal{Stage}}
\newcommand\Peel{\textnormal{Peel}}
Recall that $m_{max}=2t+\binom{t}{2}$.
For $m=2t,2t+1,\cdots, m_{max}$, do the following, which we refer to as $\Stage(m)$.
\begin{enumerate}
  \item While there exists a vertex $v\in X$ and a color $\chi\in \{b,r\}$ such that $|U\ind{\chi}_{\le m}(v,X)| \ge 10m$: 
  \begin{enumerate}
    \item Let $\hat U\ind{\chi}(v)$ be the set $U\ind{\chi}_{\le m}(v,X)$ truncated to $10m$ vertices arbitrarily.
    \item Let $B\ind{\chi}(v) = \left\{w:\abs{u\in \hat U\ind{\chi}(v): \{u,v,w\}\text{ is color $\chi$}}\le 4m\right\}.$
    \item Add $v$ to $S\ind{\chi}$.
    \item For all $u\in \hat U\ind{\chi}(v)$, add $t/m$ to $\alpha\ind{\chi}(u)$.
    \item For all $w\in B\ind{\chi}(v)$, add $\min(1/4, t/d\ind{\chi}_{vw})$ to $\beta\ind{\chi}(w)$.
    \item Delete from $X$ all vertices $u$ with $\alpha\ind{\chi}(u) \ge 1/2$ or $\beta\ind{\chi}(u) \ge 1/4$.
    \item Delete $v$ from $X$.
  \end{enumerate}
\end{enumerate}

Note that $B\ind{\chi}(v)$ and $\hat U\ind{\chi}(v)$ are only defined for $v\in S\ind{\chi}$.
We refer to steps 1(a)-1(g) as the \emph{peeling step} for $v$, denoted $\Peel(v)$.
We let $m_v$ denote the value such that the peeling step for $v$ occurred during $\Stage(m_v)$, and call $m_v$ the \emph{peeling parameter} of $v$.
Throughout the analysis, let $X_v$ denote the set $X$ immediately before $\Peel(v)$.
For any $m\in[2t,m_{max}]$, let $X_m$ denote the set $X$ immediately after $\Stage(m)$, so that $X_{m_{max}}$ is the set $X$ at the end of the peeling procedure.

The above process terminates in one of two ways.
Either we ``get stuck'', i.e. we complete Stage$(m_{max})$ and $|S\ind{b}|<t$ and $|S\ind{r}|<t$, or we ``finish'', i.e. we terminate earlier with $|S\ind{b}|=t$ or $|S\ind{r}|=t$.
We show there is a monochromatic hedgehog in each case.
In Subsection~\ref{sssec:case-1}, we handle the case where we ``get stuck''.
In Subsection~\ref{sssec:case-2}, we handle the case where we ``finish''.

\subsubsection{Basic facts about peeling}

We first establish the following facts about the procedure.
\begin{lemma}
  For any $m$ such that $2t\le m\le m_{max}$, for any time in the procedure after $\Stage(m)$, the following holds: for all colors $\chi\in\{b,r\}$, for all $m'$ with $2t\le m'\le m$, and for all vertices $v\in X$, we have $|U\ind{\chi}_{\le m'}(v,X)| < 10m'$.
\label{lem:alg-1}
\end{lemma}
\begin{proof}
  Fix $m$ with $2t\le m\le m_{max}$.
  We have $|U\ind{\chi}_{\le m}(v,X_m)| < 10m$ for all $v\in X_m$:
  if not, then there exists a vertex $v\in X_m$ with $|U\ind{\chi}_{\le m}(v,X_m)| \ge 10m$, in which case we would have peeled vertex $v$ during $\Stage(m)$, and we would have deleted $v$ from $X_m$ during $\Peel(v)$, which is a contradiction.
  Throughout the procedure, $X$ is nonincreasing.
  Thus, at any point in the procedure after $\Stage(m)$, we have $X\subset X_m$, so for all $v\in X$, we have $v\in X_m$ and $|U\ind{\chi}_{\le m}(v,X)| \le |U\ind{\chi}_{\le m}(v,X_m)| < 10m$.
\end{proof}

\begin{lemma}
  For all colors $\chi\in\{b,r\}$ and all vertices $v\in S\ind{\chi}$, we have $|B\ind{\chi}(v)|\le 2m_v$. 
\label{lem:alg-2}
\end{lemma}
\begin{proof}
  We prove this for $\chi=b$, and the case $\chi=r$ follows from symmetry.
  We double-count the number $Z$ of red triples $\{u,v,w\}$ such that $u\in \hat U\ind{b}(v)$ and $w\in B\ind{b}(v)$.
  On one hand, every $u\in \hat U\ind{b}(v)$ is in at most $m_v$ red triples because we chose $\hat U\ind{b}(v)$ as a subset of $U\ind{b}_{\le m_v}(v,X_v)$, so the total number of red triples is at most $m_v\cdot |\hat U\ind{b}(v)| = 10m_v^2$.
  On the other hand, by definition of $B\ind{b}(v)$, each $w\in B\ind{b}(v)$ is in at least $|\hat U\ind{b}(v)| - 4m_v = 6m_v$ such red triples.
  Thus, the number of such triples is at least $|B\ind{b}(v)|\cdot 6m_v$.
  Hence, $10m_v^2\ge Z \ge 6m_v |B\ind{b}(v)|$ so $|B\ind{b}(v)|\le 2m_v$ as desired.
\end{proof}

\begin{lemma}
  For all colors $\chi\in\{b,r\}$ and all vertices $v,v'\in S\ind{\chi}$, we have $d\ind{\chi}_{vv'}\ge 4t$. 
\label{lem:alg-4}
\end{lemma}
\begin{proof}
  Assume for sake of contradiction that $d\ind{\chi}_{vv'}<4t$.
  Without loss of generality, $v$ was added to $S\ind{\chi}$ before $v'$.
  We have $d\ind{\chi}_{vv'} < 4t < 4m_v$, so during $\Peel(v)$, vertex $v'$ is included in $B\ind{\chi}(v)$.
  Hence, $\min(1/4, t/d\ind{\chi}_{vv'})= 1/4$ is added to $\beta\ind{\chi}_{\le 4t}(v')$ during 1(e) of $\Peel(v)$, so during 1(f) of $\Peel(v)$, vertex $v'$ is deleted from $X$ if it hasn't been deleted already.
  Thus, we could not have added $v'$ to $S\ind{\chi}$ after $\Peel(v)$, which is a contradiction, so $d\ind{\chi}_{vv'}\ge 4t$, as desired.
\end{proof}

\subsubsection{Bounding the number of deleted vertices}

\begin{lemma}
  \label{lem:alg-7}
  For all colors $\chi\in\{b,r\}$ and all vertices $v\in S\ind{\chi}$, during $\Peel(v)$, the total increase in $\alpha\ind{\chi}(u)$ over all $u\in V$ is exactly $10t$.
\end{lemma}
\begin{proof}
  Fix $v\in S\ind{\chi}$.
  We have $|\hat U\ind{\chi}(v)|=10m_v$ by definition, and, for $u\in \hat U\ind{\chi}(v)$, each $\alpha\ind{\chi}(u)$ increases by exactly $t/m_v$, for a total increase of $10m_v\cdot (t/m_v) = 10t$.
\end{proof}

\begin{lemma}
  \label{lem:alg-3}
  For all colors $\chi\in\{b,r\}$ and all vertices $v\in S\ind{\chi}$, during $\Peel(v)$, the total increase in $\beta\ind{\chi}(w)$ over all $w\in V$ is at most $20t\ln t$.
\end{lemma}
\begin{proof}
  By symmetry, it suffices to prove the lemma for $\chi=b$.
  Let $v\in S\ind{b}$.
  For $m=0,\dots, 4m_v$, let
  \begin{align}
    a_m \ &\defeq \  \#\{w\in X_v:d\ind{b}_{vw} = m\}  \\
    a_{\le m} \ &\defeq \ a_0+a_1+\cdots+a_m = \abs{U\ind{r}_{\le m}(v,X_v)}. 
  \end{align}
  $\Peel(v)$ is after $\Stage(m_v-1)$. 
  Hence, by Lemma~\ref{lem:alg-1}, for $2t\le m \le m_v-1$, we have $a_{\le m} \le 10 m$.
  We know
  \begin{align}
    |U\ind{b}_{\le 4m_v}(v,X_v)|\ge|U\ind{b}_{\le m_v}(v,X_v)| \ge 10m_v > 8m_v,
  \end{align}
  where the second inequality holds because $v$ was chosen to be peeled in $\Stage(m_v)$.
  Hence, by Lemma~\ref{lem:alg-0}, $a_{\le 4m_v} = |U\ind{r}_{\le 4m_v}(v,X_v)|\le |U\ind{r}_{\le 4m_v}(v)|\le 8m_v$.
  As $a_{\le m}$ is non-decreasing in $m$, we conclude $a_{\le m}\le 10 m $ for $2t\le m\le 4m_v$.

  For $m=0,\dots,4m_v$, for any $w$ with $d\ind{b}_{vw}=m$, the peeling of $v$ increases $\beta\ind{b}(w)$ by exactly $\min(1/4, t/m)$.
  Thus, for $a_m$ many $w$, the penalty $\beta\ind{b}(w)$ increases by $\min(1/4, t/m)$.
  Furthermore, $\beta\ind{b}(w)$ increases only for $w\in B\ind{b}(v)$, which has at most $2m_v$ vertices by Lemma~\ref{lem:alg-2}. 
  For $2m_v - a_{\le 4m_v}$ vertices $w$, $\beta\ind{b}(w)$ increases by less than $t/4m_v$, giving a total increase in $\beta\ind{b}(w)$ of less than $t$ from those vertices.
  The total increases in $\beta\ind{b}(w)$ is thus less than
  \begin{align}
    \frac{1}{4}\left( a_0+a_1+\cdots+a_{4t} \right) + \frac{a_{4t+1}t}{4t+1} + \cdots + \frac{a_{4m_v}t}{4m_v} + t.
  \label{eq:bound-1}
  \end{align}
  The coefficients of $a_0,\dots,a_{4m_v}$ in \eqref{eq:bound-1} are nonincreasing, so \eqref{eq:bound-1} is $t$ plus a positive linear combination of $a_{\le 4t}, a_{\le 4t+1}, \cdots, a_{\le 4m_v}$.
  Subject to $a_{\le m}\le 10m$ for $2t\le m\le 4m_v$, all of $a_{\le 4t}, a_{\le 4t+1},\dots, a_{\le 4m_v}$ are simultaneously maximized when $a_0=0$ and $a_m = 10$ for $m=1,\dots,4m_v$, so \eqref{eq:bound-1} is maximized there as well.
  Hence,
  \begin{align}
    \text{Total increase in $\beta\ind{b}(w)$}
    \ &< \  \frac{1}{4}\left( a_0+a_1+\cdots+a_{4t} \right) + \frac{a_{4t+1}t}{4t+1} + \cdots + \frac{a_{4m_v}t}{4m_v} + t \nonumber\\
    \ &\le \ t + \frac{1}{4}\cdot 40t + \frac{10t}{4t+1} + \frac{10t}{4t+2}+\cdots + \frac{10t}{4m_v} \nonumber\\
    \ &\le \ 11t+10t\ln (4m_v/4t) \ < \   20t\ln t,
  \label{}
  \end{align}
  where, for the last inequality, we used $m_v\le t^2$ and $t\ge 10$.
  This is what we wanted to show.
\end{proof}

\begin{lemma}
  The total number of vertices deleted from $X$ in the peeling procedure is at most $200t^2\ln t$.
\label{lem:alg-8}
\end{lemma}
\begin{proof}
  A vertex is deleted either for being added to $S\ind{b}$ or $S\ind{r}$, having $\alpha\ind{b}(\cdot)$ or $\alpha\ind{r}(\cdot)$ at least 1/2, or having $\beta\ind{b}(\cdot)$ or $\beta\ind{r}(\cdot)$ at least 1/4. 
  At the end of the procedure, we have the following inequalities.
  For all $\chi\in\{b,r\}$ and all $u\in V$, we have $\alpha\ind{\chi}(u)$ and $b\ind{\chi}(u)$ are initially 0 and increase only during the peeling step of some vertex $v\in S\ind{\chi}$.
  Hence, by Lemma~\ref{lem:alg-7}, for $\chi\in \{b,r\}$,
  \begin{align}
    \sum_{u\in V}^{} \alpha\ind{\chi}(u) = 10t\cdot |S\ind{\chi}|\le 10t^2.
  \end{align}
  Furthermore, by Lemma~\ref{lem:alg-3}, for $\chi\in\{b,r\}$,
  \begin{align}
    \sum_{u\in V}^{} \beta\ind{\chi}(u) \le 20t\ln t\cdot |S\ind{\chi}|\le 20t^2\ln t.
  \end{align}
  We conclude that, at the end of the procedure,
  \begin{align}
    \#\{\text{deleted $u$}\}
    \ &\le \ |S\ind{b}| + |S\ind{r}| 
      + \#\{u: \text{$\alpha\ind{b}(u)\ge 1/2$}\}
      + \#\{u: \text{$\alpha\ind{r}(u)\ge 1/2$}\} \nonumber\\
      &\qquad+ \#\{u: \text{$\beta\ind{b}(u)\ge 1/4$}\} 
      + \#\{u: \text{$\beta\ind{r}(u)\ge 1/4$}\}  \nonumber\\
    \ &< \  2t + \sum_{u\in V}^{} \left(2\alpha\ind{b}(u)+2\alpha\ind{r}(u) + 4\beta\ind{b}(u) + 4\beta\ind{r}(u)\right)  \nonumber\\
    \ &\le \   2t + 2\cdot 10t^2 + 2\cdot 10t^2 + 4\cdot 20t^2\ln t + 4\cdot 20t^2\ln t  \nonumber\\
    \ &< \   200t^2\ln t. \nonumber \qedhere
  \label{}
  \end{align}
\end{proof}

\subsubsection{Case 1: Peeling procedure gets stuck}
\label{sssec:case-1}

By Lemma~\ref{lem:alg-8}, the number of vertices deleted in the peeling process is at most $200t^2\ln t$, so, at the end of the peeling procedure, $|X|\ge (200t^2\ln t + 400t^2) - 200t^2\ln t = 400t^2$.

Consider the complete 2-colored subhypergraph $\mathcal{H}'$ of $\mathcal{H}$ induced by the vertex set $X$.
By Lemma~\ref{lem:alg-1}, at the end of the procedure, for all $m=2t,2t+1,\dots,m_{max}$ and all $v\in X$,
\begin{align}
  |U\ind{b}_{\le m}(v,X)| < 10m,
  \qquad 
  |U\ind{r}_{\le m}(v,X)| < 10m.
\label{}
\end{align}
Applying Lemma~\ref{lem:bal} to $\mathcal{H}'$ with $c=10$, we conclude $\mathcal{H}'$ (and hence $\mathcal{H}$) has a red hedgehog $H_t$.\footnote{By the same reasoning $\mathcal{H}'$ also has a blue hedgehog.}

\subsubsection{Case 2: Peeling procedure finishes}
\label{sssec:case-2}

Suppose we finish with $|S\ind{b}|=t$.
The analysis for $|S\ind{r}|=t$ is symmetrical.
We try to find a blue hedgehog.
For brevity, in the rest of this section, let $S=S\ind{b}$.
Let $S=\{v_1, \cdots, v_t\}$, where the $v_i$ were chosen in the order $v_1,\dots,v_t$.
For $i=1,\dots,t$, let $m_i=m_{v_i}$ be the peeling parameter for $v_i$, so that $m_1\le m_2\le \cdots\le m_t$.

\begin{definition}
  Call a pair $v_iv_j\in\binom{S}{2}$ with $i<j$ \emph{bad} if $v_j\in B\ind{b}(v_i)$.
  Otherwise, call $v_iv_j\in\binom{S}{2}$ \emph{good}.
  Let $E_{bad} \subset \binom{S}{2}$ be the set of all bad pairs and let $E_{good}\subset\binom{S}{2}$ be the set of all good pairs, so that $\binom{S}{2} = E_{bad}\cup E_{good}$ is a partition.
\end{definition}

\begin{lemma}
  \label{lem:alg-5}
  \begin{align}
    \sum_{v_iv_j\in E_{bad}}^{} \frac{1}{d\ind{b}_{v_iv_j}} < \frac{1}{4}.
  \label{}
  \end{align}
\end{lemma}
\begin{proof}
  Fix $2\le j\le t$.
  Consider all bad pairs $v_iv_j$ with $i < j$.
  At the peeling of $v_j$, $\beta(v_j) < 1/4$, otherwise $v_j$ would have been deleted from $X$ and we could not have peeled $v_j$.
  Hence, at the peeling of $v_j$,
  \begin{align}
    \frac{1}{4}
    \ > \  \beta\ind{b}(v_j) 
    \ &= \ \sum_{\substack{i: i<j, \\ v_j\in B\ind{b}(v_i)}}^{} \min\left(\frac{1}{4},\frac{t}{d\ind{b}_{v_iv_j}}\right)
    \ = \ \sum_{\substack{i: i<j, \\ v_iv_j\in E_{bad}}}^{} \min\left(\frac{1}{4},\frac{t}{d\ind{b}_{v_iv_j}}\right)
    \ = \ \sum_{\substack{i: i<j, \\ v_iv_j\in E_{bad}}}^{} \frac{t}{d\ind{b}_{v_iv_j}}.
  \label{}
  \end{align}
  The first equality is by definition of $\beta\ind{b}(v_j)$, the second is by definition of $E_{bad}$, and the last is because $d\ind{b}_{v_iv_j}\ge 4t$ for all $i<j$ by Lemma~\ref{lem:alg-4}.
  Thus,
  \begin{align}
    \sum_{v_iv_j\in E_{bad}}^{}  \frac{1}{d\ind{b}_{v_iv_j}}
    \ &= \ \sum_{j=2}^{t}  \sum_{\substack{i:i<j, \\ v_iv_j\in E_{bad}}}^{} \frac{1}{d\ind{b}_{v_iv_j}}
    \ \le \  \sum_{j=2}^{t} \frac{1}{4t} \ < \ \frac{1}{4}. \nonumber\qedhere
  \label{}
  \end{align}
\end{proof}

We prove that there is a blue hedgehog with body $S$, by showing the matching condition of Lemma~\ref{lem:hall} holds.
Consider an arbitrary $F\subset \binom{S}{2}$.
Partition $F = F_{bad}\cup F_{good}$, where $F_{bad} = F\cap E_{bad}$ and $F_{good} = F\cap E_{good}$.
We wish to show that $N\ind{b}(F)\ge |F|+t$.

\textbf{Subcase 1: $|F_{bad}|\ge |F_{good}|$}.

By Lemma~\ref{lem:alg-5},
\begin{align}
  \frac{|F_{bad}|}{\max_{v_iv_j\in F_{bad}}d\ind{b}_{v_iv_j}}
  \ \le \ \sum_{v_iv_j\in F_{bad}}^{} \frac{1}{d\ind{b}_{v_iv_j}}
  \ \le \ \sum_{v_iv_j\in E_{bad}}^{} \frac{1}{d\ind{b}_{v_iv_j}}
  \ < \ \frac{1}{4} .
\label{}
\end{align}
Thus, there exists some $v_iv_j\in F_{bad}$ such that $d\ind{b}_{v_iv_j} > 4|F_{bad}|$.
Furthermore, this $v_iv_j$ satisfies $d\ind{b}_{v_iv_j}\ge 4t$ by Lemma~\ref{lem:alg-4}, so $d\ind{b}_{v_iv_j}\ge 2|F_{bad}|+ 2t$.
Hence, 
\begin{align}
  |N\ind{b}(F)|
  \ \ge \ d\ind{b}_{v_iv_j}
  \ \ge \ 2|F_{bad}| + 2t
  \ \ge \ |F_{bad}| + |F_{good}| + 2t 
  \ > \ |F|+t,
\end{align}
as desired.
The first inequality is because the blue edges containing $v_iv_j$ are all elements of $N\ind{b}(F)$.
The second inequality is because $d\ind{b}_{v_iv_j}$ is at least $4|F_{bad}|$ and at least $4t$ by above.
The third inequality is by the assumption $|F_{bad}|\ge |F_{good}|$.
The fourth inequality is because $|F| = |F_{bad}| + |F_{good}|$ and $2t>t$.

\textbf{Subcase 2: $|F_{bad}| < |F_{good}|$.}

In particular, $|F_{good}| > 0$, so $|F|$ has some good pair $v_iv_j$ with $i<j$.
This pair is in at least $4m_i \ge 8t$ blue triples, so $|N\ind{b}(F)|\ge 8t$.

Let $I$ be the set of all indices $i$ such that there exists $j$ with $i<j\le t$ with $v_iv_j\in F_{good}$.
For each $i$, there are less than $t$ indices $j$ such that $i < j\le t$, so
\begin{align}
  \label{eq:alg-9}
  |I|\cdot t > |F_{good}|.
\end{align}
For each $i\in I$, arbitrarily fix $j_i>i$ such that $v_iv_{j_i}$ is good.
For $i\in I$, define
\begin{align}
  U^*_i \defeq N\ind{b}(\{v_iv_{j_i}\})\cap \hat U\ind{b}(v_i), \qquad
  U^*_I \defeq \bigcup_{i\in I}^{} U^*_i.
\end{align}
so that $U^*_I \subset N\ind{b}(F)$.
For all $i\in I$, the pair $v_iv_{j_i}$ is good, so $v_{j_i}\notin B\ind{b}(v_i)$.
Hence, by the definition of $B\ind{b}(v_i)$, there are more than $4m_i$ vertices $u\in \hat U\ind{b}(v_i)$ such that $\{u,v_i,v_{j_i}\}$ is blue.
Thus, for all $i\in I$, the set $U^*_i$ has at least $4m_i$ vertices.
In the peeling of $v_i$, the penalty $\alpha\ind{b}(u)$ increases by $t/m_i$ for each $u\in U^*_i$.
Hence, in peeling $v_i$, the sum of penalties $\sum_{u\in U^*_I}^{} \alpha\ind{b}(u)$, increases by at least $4m_i\cdot t/m_i = 4t$.
Thus,
\begin{align}
  4t\cdot |I| \ \le \  \sum_{u\in U^*_I}^{} \alpha\ind{b}(u).
\label{eq:alg-10}
\end{align}
On the other hand, the vertex $u$ is deleted from $X$ whenever $\alpha\ind{b}(u)\ge 1/2$, the penalty $\alpha\ind{b}(u)$ increases by at most $t/2t = 1/2$ in any peeling step, and the penalty $\alpha\ind{b}(u)$ never changes after $u$ is deleted from $X$.
Thus, for all vertices $u\in V$, we have
\begin{align}
  \label{eq:alg-11}
  \alpha\ind{b}(u) \le 1.
\end{align}
We conclude
\begin{align}
  2|F|
  \ &\le \  4|F_{good}|
  \ \le \  4t|I| 
  \ \le \  \sum_{u\in U^*_I}^{} \alpha\ind{b}(u)  \nonumber\\
  \ &\le \  \sum_{u\in U^*_I}^{} 1  
  \ = \ \abs{U_I^*} 
  \ \le \ \abs{N\ind{b}(F_{good})} 
  \ \le \ \abs{N\ind{b}(F)}.
\label{}
\end{align}
The first inequality is by the assumption $|F_{bad}|<|F_{good}|$, the second is by \eqref{eq:alg-9}, the third is by \eqref{eq:alg-10}, the fourth is by \eqref{eq:alg-11}, the fifth is by $U_I^*\subset N\ind{b}(F_{good})$, and the sixth is by $F_{good}\subset F$.
Combining with $\abs{N\ind{b}(F)}\ge 8t$, we conclude $\abs{N\ind{b}(F)}\ge |F|+t$, as desired.

This covers all subcases, so we've proven that, for any nonempty subset $F \subset \binom{S}{2}$, we have $N\ind{b}(F)\ge |F|+t$. Hence, the matching condition of Lemma~\ref{lem:hall} holds, so there is a blue hedgehog with body $S$, as desired.
This completes the proof of Theorem~\ref{thm:main}.
\qed

%% file: hedgehog_main.bbl
\begin{thebibliography}{}

\bibitem{AKS} N. Alon, M. Krivelevich and B. Sudakov, Tur\'an numbers of bipartite graphs and related Ramsey-type questions, {\it Combin. Probab. Comput.} {\bf 12} (2003), 477--494. 

\bibitem{BE75}
{S. A. Burr and P. Erd\H{o}s,} {On the magnitude of generalized Ramsey numbers for graphs,} {in Infinite and Finite Sets, Vol. 1 (Keszthely, 1973), 214--240, Colloq. Math. Soc. J\'anos Bolyai, Vol. 10,} North-Holland, Amsterdam, 1975.

\bibitem{CFS10}
{D. Conlon, J. Fox and B. Sudakov,} {Hypergraph Ramsey numbers,} {\it J. Amer. Math. Soc.} {\bf 23} (2010), 247--266.

\bibitem{CFS15} {D. Conlon, J. Fox and B. Sudakov,} {Recent developments in graph Ramsey theory,} {\it Survey in Combinatorics 2015}  49--118, London Math. Soc. Lecture Note Ser., 424, Cambridge Univ. Press, Cambridge, 2015.

\bibitem{CFR15}
{D. Conlon, J. Fox and V. R{\"o}dl,} {Hedgehogs are not colour blind,} {\it J. Comb.} {\bf 8} (2017), 475--485. 

\bibitem{EHR65}
{P. Erd\H{o}s, A. Hajnal and R. Rado,} {Partition relations for cardinal numbers,} {\it Acta Math. Acad. Sci. Hungar.} {\bf 16} (1965), 93--196. 

\bibitem{ER52}
{P. Erd\H{o}s and R. Rado,} {Combinatorial theorems on classifications of subsets of a given set,} {\it Proc. London Math. Soc.} {\bf 3} (1952), 417--439. 

\bibitem{FS09}
{J. Fox and B. Sudakov,} {Two remarks on the Burr--Erd\H{o}s conjecture,} {\it European J. Combin.} {\bf 30} (2009), 1630--1645.

\bibitem{GRS90}
{R. L. Graham, B. L. Rothschild and J. H. Spencer,} {\bf Ramsey theory}, 2nd edition, {John Wiley \& Sons}, 1990. 

\bibitem{KR06}
{A. V. Kostochka and V. R\"odl,} {On Ramsey numbers of uniform hypergraphs with given maximum degree,} {\it J. Combin. Theory Ser. A} {\bf 113} (2006), 1555--1564.

\bibitem{KS03}
{A. V. Kostochka and B. Sudakov,} {On Ramsey numbers of sparse graphs,} {\it Combin. Probab. Comput.} {\bf 12} (2003), 627--641.

\bibitem{L15}
{C. Lee,} {Ramsey numbers of degenerate graphs,} {\it Ann. of Math.} {\bf 185} (2017), 791--829. 

\end{thebibliography}
